\documentclass[reqno]{amsart}
\usepackage{enumitem}
\usepackage[all]{xy}
\usepackage{newlfont}
\usepackage[centertags]{amsmath}
\usepackage{amsfonts}
\usepackage{amssymb}
\usepackage{plain}
\usepackage{mathrsfs}

\theoremstyle{definition}
\newtheorem{theorem}{Theorem.}[section]

\newtheorem{definition}[theorem]{Definition.}
\newtheorem{example}[theorem]{Example.}
\newtheorem{remark}[theorem]{Remark.}

\begin{document}
\title{Observational Banach Manifolds}
\author[Mohammad Mehrpooya]{Mohammad Mehrpooya${}^1$}
\author[Mohammadreza Molaei]{Mohammadreza Molaei${}^2$}
\date{}
\dedicatory{\textit{\normalsize Department of Mathematics, Faculty of Mathematics and Computer\\Shahid Bahonar University of Kerman, 7616914111, Kerman, Iran}\\
{\normalfont ${}^1$\texttt{sdmehrpooya@gmail.com} \ \  ${}^2$\texttt{mrmolaei@uk.ac.ir}}}
\maketitle

\begin{abstract} %\normalsize
In this paper, the concept of selective real manifolds is extended. It is proved that the product of two selective Banach manifolds is a selective Banach manifold. The notion of the $\alpha$--level differentiation of the mappings between selective Banach manifolds is presented. Basic properties of $(r, \alpha)$--differentiable maps are studied. Tangent space for at a point of a selective Banach manifold is considered.
\end{abstract}

\section{Introduction and preliminaries}

%\indent
The problem of finding a mathematical model that can change the signature of a metric on a manifold is an interesting topic for both physicists and mathematicians. It was remained unsolved for a long time \cite{DD12, DD2, DD9} and it has been considered from observer viewpoint in \cite{DD2}. The mathematical model of one dimensional observer has been considered first in \cite{DD20}. In 2006, the second author of this article put forward the idea of the relative manifold by considering the notion of one dimensional observer on a manifold, and adopted a realistic approach to the problem of unity using the concept of relative metrics over the relative manifolds \cite{DD4}. In \cite{DD5}, the notion of multi-dimensional observers introduced, and it is employed to prove a version of Tychonoff Theorem and new concept of topological entropy \cite{DD5, DD11, DD1}. In \cite{DD8}, the concept of synchronization for continuous time dynamical systems from the viewpoint of an observer has been considered. This notion is a generalization of synchronization, and it is proved that the future of the points of the set in which two dynamical systems are relative probability synchronized is the same up to the homeomorphism determined by a relative probability synchronization \cite{DD8}.\\
The idea of a relative metric space, as a mathematical model compatible with a physical phenomena has been considered in 2011 \cite{DD3}. In \cite{DD3} the concept of relative topological entropy for relative semi-dynamical systems on a relative metric space has been  studied. The notion of selective manifolds on $\mathbb{R}^n$, as the suitable finite space for the problem of unity has been presented in 2011 \cite{DD7}.

In this paper, we put forward the notion of a selective manifold over a Banach space \cite{DD33} and we study its characteristics. In this respect, we introduce the idea of a selective manifold on a Banach space in section 2. We prove that the product of two selective Banach manifolds is a selective Banach manifold. Next, we present the concept of $\alpha$--level differentiation of the mappings between selective manifolds and we prove a new version of chain rule theorem for the mappings between selective Banach manifolds in section 3. Then, we study the notion of the tangent space of a selective Banach manifold at a given point in section 4. In the rest of this section, we provide some preliminaries on observers.

Let $M$ be a set. By an observer of dimension $m$ on $M$, we mean a mapping $\mu: M \longrightarrow \prod_{j \in J} I_j$, where $I_j=[0, 1]$ for every $j \in J$ and $CardJ=m$. From a physical point of view, an observer considers finitely many physical parameters like speed, energy, etc. Thus, the product of $[0, 1]$ is used in order to define an observer mathematically. The observational models are widely applied in biology, dynamical systems, geometry \cite{DD4, DD6, DD11, DD1}.

Let $\mu: M \longrightarrow \prod_{j \in J} I_j$ and $\eta: M \longrightarrow \prod_{j \in J} I_j$ be two observers of dimension $m$, where $I_j=[0, 1]$ for all $j \in J$. By $\eta \subseteq \mu$, we mean that $\eta_j(x) \leqslant \mu_j(x)$ for all $x \in M$ and $j \in J$. We write $\eta \subset \mu$ if $\eta \subseteq \mu$ and for given $x \in M$, there is $j\in J$ $\eta_j(x) < \mu_j(x)$.

Let $\mu$ and $\eta$ be two observers of dimension $m$ on the set $M$. We define two $M$ dimensional observers $\mu \bigcup \eta$ and $\mu \bigcap \eta$ by:
\begin{equation*}
(\mu \bigcup \eta)_j(x)=\sup\{\mu_j(x), \eta_j(x)\};
\end{equation*}
\begin{equation*}
(\mu \bigcap \eta)_j(x)=\inf\{\mu_j(x), \eta_j(x)\}.
\end{equation*}

Let $\tau_\mu$ be a collection of subsets of $\mu$ for which the following conditions are satisfied:
\begin{enumerate}[label=a\arabic*),itemindent=1cm,ref=a\arabic*]
\item $\mu \in \tau_\mu$ and $\chi_\phi^m \in \tau_\mu$, where $\chi_\phi^m(x)=\prod_{j \in J}0$;
\item $\lambda \bigcap \eta \in \tau_\mu$, whenever $\lambda, \eta \in \tau_\mu$;
\item if $\{\eta_\alpha\}_{\alpha \in \varLambda}$ is any collection of $\tau_\mu$, then $\bigcup_{\alpha \in \Lambda} \in \tau_\mu$, where $\Lambda$ is an index set.
\end{enumerate}
The collection $\tau_\mu$ is called the $\mu$--topology on $M$.

\section{The $\mu$--Selective Banach manifolds}

Let $\lambda$ be an observer of dimension $m$ on $M$. $\lambda$ is called the constant observer if there exists $j \in J$ such that $\lambda_j(x)$ is constant for all $x \in M$. The collection of all constant observers is denoted by $C$. Suppose that $r$ and $s$ be two constant observers with $r < s$. If $\lambda \in \tau_\mu$, then we define
\begin{equation*}
\lambda^{-1}(r, s) \text{ by } \{x \in M \mid r_j(x) < \lambda_j(x)< s_j(x) \quad \forall j \in J\}.
\end{equation*}
We use the notation $(\tau_\mu)_r^s$ to denote the topology of $M$ is generated by $\{\lambda^{-1}(r, s)\mid\lambda \in \tau_\mu\}$ over $M$.

\begin{definition}
Let $K_1 \subseteq \bigcup_{r, s \in C}(\tau_\mu)_r^s$, and $K_1\neq \emptyset$. Put
\begin{equation*}
K=\{U\mid U=U_1 \bigcap \mu^{-1}(r, s) \quad \& \quad r, s \in C \quad \& \quad U_1 \in K_1\}.
\end{equation*} By a chrat for $M$, we mean a pair $(U, \phi)$, where $U \in K$ and $\phi$ is a one to one mapping over $U$ such that $\phi(U)$ is a $C^n$ Banach manifold. Put
\begin{equation*}
D=\{(U, \phi)\mid(U, \phi) \text{ is a chart for } M\}.
\end{equation*}
$D$ is called a $C^n$ $\mu$--structure if
\begin{enumerate}[label=b\arabic*),itemindent=1cm,ref=b\arabic*]
\item $\mu(M) = \mu(\bigcup_{U \in K}U)$;
\item if $(U, \phi) \in D$ and $(V, \psi) \in D$ for which $\mu(U) = \mu(V)$, then there exist a one to one and onto map $h_0: U \longrightarrow V $ and a $C^n$ Banach diffeomorphism $h:\phi(U)\longrightarrow\psi(V)$ such that $ho\phi = \psi oh_0$, where the restriction of $h_0$ to $U \bigcap V$ is the identity map.
\end{enumerate}
\end{definition}

\begin{example}
Let $M$ be a compact smooth manifold of dimension $m$. Put
\begin{equation*}
N = \{(f, p)\mid f:M \longrightarrow M \text{ is smooth and } p \text{ is a hyperbolic fixed point for } f\}
\end{equation*}
We define
\begin{equation*}
\mu: N \longrightarrow [0, 1]\times[0, 1]
\end{equation*}
\begin{equation*}
\mu(f, p) = (\dfrac{1}{\sigma_p(f)+1}, \dfrac{1}{\delta_p(f)+1});
\end{equation*}
where
\begin{equation*}
\sigma_p(f) = Card(\{\lambda \mid \lambda \text{ is an eigenvalue of } f \text{ at the hyperbolic fixed point } p \quad \& \quad \| \lambda \| > 1\});
\end{equation*}
and
\begin{equation*}
\delta_p(f) = Card(\{\lambda \mid \lambda \text{ is an eigenvalue of } f \text{ at the hyperbolic fixed point } p \quad \& \quad \| \lambda \| < 1\}).
\end{equation*}
Suppose that $r$ and $s$ be two arbitrary constant observers on $N$, $r < s$, and $\varUpsilon = \{\lambda^{-1}(a, b)\mid\lambda \in \tau_\mu\}$. Let $(\tau_\mu)_r^s$ denote the topology generated by $\varUpsilon$ over $N$. Put
\begin{equation*}
K = \{U\mid U=U_0\cap \mu^{-1}(r, s) \quad \& \quad r,s \in C \quad \& \quad U_0 \in K_1\}.
\end{equation*}
We define
\begin{equation*}
\phi: U\longrightarrow L(\mathbb{R}^m, \mathbb{R}^m)
\end{equation*}
\begin{equation*}
(f, p)\longmapsto df(p);
\end{equation*}
where $p$ is a hyperbolic fixed point for $f$. Then, $D = \{(U, \phi)\mid U \in K\}$ is a smooth $\mu$--structure for $N$, and $(N, D)$ is a $\mu$--selective Banach manifold.
\end{example}

\begin{definition}
Let $E$ be a Banach space. $(M, D)$ is called a $\mu$--selective Banach manifold modeled over $E$ if $D$ is a $C^n$ $\mu$--structure.
\end{definition}
Let $(M, D_i)$ be  $C^n$ $\mu_i$--selective Banach manifold for $i \in\{1, 2\}$. We write $D_1 \leq D_2$ if there exists a one to one map $f: \bigcup_{(U, \phi) \in D_1}U\longrightarrow\bigcup_{(V, \psi) \in D_2}V$ such that $\mu_1(U) = \mu_2(f(U))$ and $(f(U), \phi of^{-1}) \in D_2$, where $(U, \phi) \in D_1$. $D_1$ and $D_2$ are called equivalent if $D_1 \leq D_2$ and $D_2 \leq D_1$, where $D_i$ is $C^n$ $\mu_i$--structure, and $i = 1, 2$.

\begin{theorem}\label{item: 1.3}
Let $(M, D)$ be a $C^n$ $\mu$--selective Banach manifold. There exists a $C^n$ $\mu$--structure, $A$, on $M$ such that every $C^n$ $\mu$--structure, $D_0$, on $M$ with $A \leq D_0$ is equivalent to $A$.
\end{theorem}

\begin{proof}
If $\varLambda$ is the set of $C^n$ $\mu$--structures of $M$, and $\{D_1, D_2, \ldots\}$ is a chain of elements of $\Lambda$, then $H=\bigcup_{i\in \mathbb{N}}D_i$ is a $C^n$ $\mu$--structure. We only proved the second axiom: if $(U, \phi), (V, \psi)\in H$ and $\mu(U)=\mu(V)$, then there exist $i\leq j$ such that $(U, \phi)\in D_i$ and $(V, \psi)\in D_j$. Moreover, there is an injective mapping $f$ such that $(f(U), \phi of^{-1})\in D_j$. Since $\mu(f(U))=\mu(V)$, there exists an injective mapping $\eta_0$ from $f(U)$ to $V$ and a $C^n$ difeomorphism $\eta$ from $\psi(U)$ to $\phi(V)$ such that $\phi o\eta_0 of=\eta o \phi$. These properties of $\eta_0 o f$ and $\eta$ imply that axiom two is valid for $(U, \phi)$ and $(V, \psi)$. Thus, each chain has a maximal element. Therefore, Zorn's Lemma implies that $\Lambda$ has at least one maximal element we call it $A$.
\end{proof}

Any $C^n$ $\mu$--structure, $A$, that satisfied the conditions of Theorem \ref{item: 1.3} is called a $C^n$ maximal $\mu$--structure.

\begin{definition}
$(M, A, \mu)$ is called $C^n$ selective Banach manifold if $A$ is a $C^n$ maximal $\mu$--structure on $M$ and $\mu(M \smallsetminus \bigcup_{(U, \phi) \in A}U) = \{0\}$. In this case, $A$ is called a $C^n$ $\mu$--atlas.
\end{definition}

\begin{theorem}\label{item: 1.5}
Let $(M_i, D_i)$ be $C^n$ $\mu_i$--selective Banach manifold for $i \in \{1, 2\}$. Then, \\ $(M_1\times M_2, \varOmega)$ is a $C^n$ $\mu$--selective Banach manifold, where
\begin{equation*}
\varOmega = \{(U_\alpha\times V_\beta, \phi_\alpha\times \psi_\beta) \mid (U_\alpha, \phi_\alpha) \in D_1 \quad \& \quad (V_\beta, \psi_\beta) \in D_2\};
\end{equation*}
whenever
\begin{equation*}
\phi_\alpha \times \psi_\beta: U_\alpha \times V_\beta\longrightarrow (\phi_\alpha \times  \psi_\beta)(U_\alpha \times V_\beta)
\end{equation*}
\begin{equation*}
(\phi_\alpha \times \psi_\beta)(x, y) = (\phi_\alpha(x), \psi_\beta(y));
\end{equation*}
and
\begin{equation*}
\mu: M_1 \times M_2\longrightarrow [0, 1]\times [0, 1]
\end{equation*}
\begin{equation*}
\mu(x, y) = (\mu_1(x), \mu_2(y)).
\end{equation*}
\end{theorem}

\begin{proof}
First, we prove
\begin{equation*}
\mu(M_1 \times M_2) = \mu (\bigcup_{\alpha, \beta \in \Lambda}(U_\alpha \times V_\beta)).
\end{equation*}
Let $(x, y) \in M_1 \times M_2$ be arbitrary. We have $\mu(x, y) = (\mu_1(x), \mu_2(y))$. Since
\begin{equation*}
\mu_1(x) \in \mu_1(\bigcup_{\alpha \in\Lambda}U_{\alpha}) \quad \& \quad \mu_2(y) \in \mu_2(\bigcup_{\beta \in\Lambda}V_{\beta});
\end{equation*}
then,
\begin{equation*}
\mu(x, y) \in \mu(\bigcup_{\alpha \in\Lambda}(U_\alpha \times V_\beta)).
\end{equation*}
Thus,
\begin{equation}\label{item: 1}
\mu(M_1 \times M_2) \subseteq \mu(\bigcup_{\alpha, \beta \in\Lambda}(U_\alpha \times V_\beta)).
\end{equation}
On the other hand, we have
\begin{equation}\label{item: 2}
\mu(\bigcup_{\alpha, \beta \in\Lambda}(U_\alpha \times V_\beta)) \subseteq \mu(M_1 \times M_2).
\end{equation}
Thus,
\begin{equation*}
\mu(M_1 \times M_2) = \mu(\bigcup_{\alpha, \beta \in\Lambda}(U_\alpha \times V_\beta)).
\end{equation*}
Now, we show the second axiom is satisfied. Let $(U_\alpha \times V_\beta,\phi_\alpha \times \psi_\beta), (U_\sigma \times V_\delta,\phi_\sigma \times \psi_\delta) \in K$ for which we have
\begin{equation*}
\mu(U_\alpha \times V_\beta) = \mu(U_\sigma \times V_\delta);
\end{equation*}
then, we obtain
\begin{equation*}
(\mu_1(U_\alpha), \mu_2 (V_\beta)) = (\mu_1(U_\sigma), \mu_2 (V_\delta)).
\end{equation*}
Since $\mu_1(U_\alpha) = \mu_1(U_\sigma)$, there exists a one to one and onto mapping $\xi_0: U_\alpha\longrightarrow U_\sigma$ and a $C^n$ diffeomorphism $\xi: \phi_\alpha(U_\alpha) \longrightarrow \phi_\sigma(U_\sigma)$ such that
\begin{equation}\label{item: 3}
\xi o\phi_\alpha = \phi_\sigma o \xi_0.
\end{equation}
Since $\mu_2(V_\beta) = \mu_2(V_\delta)$, then
\begin{equation}\label{item: 4}
\eta o\psi_\beta = \psi_\delta o \eta_0;
\end{equation}
for some one to one and onto mapping $\eta_0: V_\beta \longrightarrow V_\delta$ and some $C^n$ diffeomorphism\\ $\eta: \psi_\beta(V_\beta)\longrightarrow\psi_\delta(V_\delta)$.
Now, we define
\begin{equation*}
\xi_0 \times \eta_0: U_\alpha \times V_\beta \longrightarrow U_\sigma \times V_\delta
\end{equation*}
\begin{equation*}
(\xi_0 \times \eta_0)(u, v) = (\xi_0(u), \eta_0(v)).
\end{equation*}
The bijectivity of $\xi_0$ and $\eta_0$ implies to the bijectivity of $\xi_0 \times \eta_0$.

Also, we define
\begin{equation*}
\xi \times \eta: \phi_\alpha(U_\alpha) \times \psi_\beta(V_\beta) \longrightarrow \phi_\sigma(U_\sigma) \times \psi_\delta(V_\delta)
\end{equation*}
\begin{equation*}
(\xi \times \eta)(x, y) = (\xi(x), \eta(y)).
\end{equation*}
Notice that $\xi \times \eta$ is a diffeomorphism since $\xi$ and $\eta$ are diffeomorphisms.

Finally, we show that the following diagram commutes.

\[\xymatrix{
U_\alpha\times V_\beta \ar[d]_{\phi_\alpha\times\psi_\beta} \ar[r]^{\xi_0\times \eta_0} & U_\sigma\times V_\delta \ar[d]^{\phi_\sigma\times\psi_\delta} \\
\phi_\alpha(U_\alpha)\times \psi_\beta(V_\beta) \ar[r]_{\xi\times\eta} &  \phi_\sigma(U_\sigma)\times \psi_\delta(V_\delta)}
\]

In fact, for arbitrary $(u, v) \in U_\alpha \times V_\beta$, we have
\begin{align*}
[(\xi \times \eta)o(\phi_\alpha \times \psi_\beta)](u, v) &= (\xi \times \eta)(\phi_\alpha(u), \psi_\beta(v)) \\&= (\xi(\phi_\alpha(u)), \eta(\psi_\beta(v))) \\&= ((\phi_\sigma o \xi_0)(u), (\psi_\delta o\eta_0)(v)) \\& = (\phi_\sigma \times \psi_\delta)(\xi_0(u), \eta_0(v)) \\& = [(\phi_\sigma \times \psi_\delta) o (\xi_0 \times \eta_0)](u, v). \qedhere
\end{align*}
\end{proof}

By use of the above theorem and induction, we deduce the following theorem:
\begin{theorem}
Let $(M_i, D_i)$ be $C^n$ $\mu_i$--selective Banach manifold for $i \in \{1, \ldots, m\}$. Then, $(\prod_{j=1}^mM_j, \varOmega)$ is a $C^n$ $\mu$--selective Banach Manifold, where
\begin{equation*}
\varOmega = \{(\prod_{j=1}^mU_{\alpha_j}, \prod_{j=1}^m\phi_{\alpha_j})\mid (U_{\alpha_j}, \phi_{\alpha_j}) \in D_j, \quad j=1, \ldots, m\};
\end{equation*}
\begin{equation*}
\prod_{j=1}^m\phi_{\alpha_j}:\prod_{j=1}^mU_{\alpha_j} \longrightarrow (\prod_{j=1}^m\phi_{\alpha_j})(\prod_{j=1}^mU_{\alpha_j}) \quad \text{ with } \quad
(\prod_{j=1}^m\phi_{\alpha_j})(x_{\alpha_1}, \ldots, x_{\alpha_m}) = \prod_{j=1}^m\phi_{\alpha_j}(x_{\alpha_j});
\end{equation*}
and
\begin{equation*}
\mu: \prod_{j=1}^mM_j \longrightarrow \prod_{j=1}^mI_j \quad \text{ with } \quad \mu(x_1, \ldots, x_m) = \prod_{j=1}^m\mu_j(x_j);
\end{equation*}
whenever $I_j = [0, 1]$ for all $j \in \{1, \ldots, m\}$.
\end{theorem}

\section{The $\alpha$--level differentiation of maps between selective Banach manifolds}

In this section, the concept of the $(r, \alpha)$--differentiation of the mappings between selective Banach manifolds is presented, and its characteristics are studied. Specifically, a new version of chain rule theorem for the composition of mappings between selective Banach manifolds is proved.

Let $(M_i, D_{M_i})$ be a $C^n$ $\mu_i$--selective Banach manifold that is modeled over $E_i$ for $i \in \{1, 2\}$. For every $\alpha \in [0, 1]$, and every mapping $f: M_1\longrightarrow M_2$, put
\begin{equation*}
K_{f, \alpha}=\{(U, \phi; V, \psi) \mid (U, \phi) \in D_{M_1} \quad \& \quad (V, \psi) \in D_{M_2} \quad \& \quad f(U) \subseteq V \quad \& \quad \mu(U)=\gamma(V)=\alpha\};
\end{equation*}
where $D_{M_i}$ is a $C^n$ $\mu_i$--structure on $M_i$, and $E_i$ is a Banach space for $i\in\{1, 2\}$.
\begin{definition}\label{3 1}
Let the mapping $f: M_1 \longrightarrow M_2$ be as above, $p \in M_1$, and $r \geq 0$. We say that $f$ is $(r, \alpha)$--differentiable at $p$ if the following conditions are satisfied:
\begin{enumerate}[label=c\arabic*),itemindent=1cm,ref=c\arabic*]
\item $K_{f, \alpha} \neq \emptyset$;
\item $\mu_2 of = \mu_1$; \label{a2}
\item if $(V, \psi) \in D_{M_2}$ and $\mu_2(V)=\alpha$, then $\gamma(V)=\mu_1(f^{-1}(V))$; \label{a3}
\item if $(U, \phi; V, \psi) \in K_{f, \alpha}$, then the mapping $\psi o f o \phi^{-1}: \label{a4} \phi(U)\longrightarrow\psi(V)$ is a $C^r$--map in a neighborhood of $\phi(p)$.
\end{enumerate}
\end{definition}

Condition (\ref{a4}) of the above definition implies that the definition of the $(r, \alpha)$--differentiable map is independent of the choice of the charts.
\begin{definition}
We say that the mapping $f$ is continuous at $p$ if $\psi o f o \phi^{-1}$ is a continuous map at $\phi(p)$.
\end{definition}

\begin{remark}
If $f$ is $(r+1, \alpha)$--differentiable, then $f$ is $(r, \alpha)$--differentiable.
\end{remark}

\begin{definition}
We say that the mapping $f$ is $(r, \alpha)$--differentiable on $M$ if $f$ is $(r, \alpha)$--differentiable at every $p \in M$. If $r = \infty$, then we say that $f$ is $\alpha$--smooth.
\end{definition}

\begin{theorem}
Let $(M, D_M)$ be a $C^n$ $\mu$--selective Banach manifold. Then, the identity mapping $id: M\longrightarrow M$ is $(r, \alpha)$--differentiable for all $\alpha \in Im(\mu)$, where $id(m) = m$ for all $m \in M$.
\end{theorem}

\begin{proof}
Let $\alpha \in Im(\mu)$. Suppose that $(U, \phi) \in D_M$ for all $m \in M$, where $m \in U$, and $\mu(U) = \alpha$, then $id(U) = U$. Thus, $K_{id, \alpha} \neq \phi$ and $\mu o (id) = \mu$. Let $(V, \psi) \in D_M$, and $\mu(V) = \alpha$. Since $(id)^{-1}(V) = V$, then $\mu(V) = \mu((id)^{-1}(V)$. Finally, the mapping $\phi o (id) o \phi^{-1}: \phi(U)\longrightarrow \phi(U)$ is differentiable since it is the identity mapping on a Banach space.
\end{proof}

\begin{theorem}\label{item: 3 6}
Let $(M_i, D_{M_i})$ is a $C^n$ $\mu_i$--selective Banach manifold for $i \in \{1, 2\}$, and\\ $f: M_1\longrightarrow M_2$ is a map. If $(V, \psi) \in D_{M_2}$, then there exists $(U, \phi) \in D_{M_1}$ such that $f^{-1}(V) = U$.
\end{theorem}

\begin{proof}
Let $(V, \psi) \in D_{M_2}$, then $V = V_1 \bigcap \mu_2^{-1}(r, s)$, where $r, s$ are two constant observers on $M_2$, and $V_1 = \lambda^{-1}(r_0, s_0)$, whenever $r_0$ and $s_0$ are two constant observers on $M_2$, and $\lambda \in \tau_{\mu_2}$. Thus,
\begin{align*}
f^{-1}(V) & = f^{-1}(V_1) \bigcap f^{-1}(\mu_2^{-1}(r, s)) \\& = f^{-1}(V_1) \bigcap (\mu_2 o f)^{-1}(r', s') \\& = f^{-1}(V_1) \bigcap (\mu_1)^{-1}(r', s');
\end{align*}
where
\begin{align*}
f^{-1}(\mu_2^{-1}(r, s)) & = \{f^{-1}(x)\mid r(x) < \mu_2(x) < s(x)\quad \& \quad x \in M_2\} \\& = \{f^{-1}(x)\mid (r o f o f^{-1})(x) < (\mu_2 o f o f^{-1})(x) < (s o f o f^{-1})(x) \quad \& \quad x \in M_2\}.
\end{align*}
Put $rof=r'$, and $sof=s'$. Then, $r'$ and $s'$ are two constant observers on $M_1$.

We claim that $f^{-1}(V_1)=U_1$, where $U_1=\beta^{-1}(m_0, n_0)$, $\beta \in \tau_{\mu_1}$, and $m_0$ and $n_0$ are constant observers on $M_1$. Notice that
\begin{equation*}
f^{-1}(V_1) = f^{-1}(\lambda^{-1}(r_0, s_0)) = (\lambda o f)^{-1}(m_0, n_0);
\end{equation*}
where
\begin{align*}
 f^{-1}(\lambda^{-1}(r_0, s_0)) & = \{f^{-1}(x)\mid r_0(x)<\lambda(x)<s_0(x)\quad \& \quad x \in M_2\} \\& = \{f^{-1}(x)\mid (r_0 o f o f^{-1})(x)<(\lambda o f o f^{-1})(x)<(s_0 o f o f^{-1})(x) \quad \& \quad x \in M_2\}.
\end{align*}
Put $\beta=\lambda o f$, $m_0=r_0of$, and $n_0=s_0of$. Then, for constant observers $m_0$ and $n_0$ on $M_1$, we have
\begin{equation*}
f^{-1}(V_1)=\beta^{-1}(m_0, n_0).
\end{equation*}
Since $\lambda\subseteq\mu_2$, we get $\lambda o f \subset \mu_2 o f$. Thus, $\beta \in \tau_{\mu_1}$.
\end{proof}

Here, we establish chain rule Theorem for the composition of the $(r, \alpha)$--differentiable maps between $\mu_i$--selective Banach manifolds.

\begin{theorem}\label{item: 3 7}
Let $(M_i, D_{M_i})$ be a $C^n$ $\mu_i$--selective Banach manifold over a Banach space $E_i$ for $i\in \{1, 2, 3\}$, and $f: M_1\longrightarrow M_2$ and $g:M_2\longrightarrow M_3$ be two $(r, \alpha)$--differentiable maps. Then, the mapping $gof: M_1\longrightarrow M_3$ is $(r, \alpha)$--differentiable.
\end{theorem}

\begin{proof}
Let $p \in M_1$ be given. Since $g: M_2\longrightarrow M_3$ is $(r, \alpha)$--differentiable, we can choose the charts $(V, \tilde{\phi}) \in D_{M_2}$ and $(W, \tilde{\tilde{\phi}}) \in D_{M_3}$ such that $g(f(p)) \in W$, $f(p) \in V$, $g(V) \subseteq W$, and
\begin{equation}\label{5}
\mu_2(V)=\mu_3(W)=\alpha.
\end{equation}
It follows from Theorem \ref{item: 3 6} that for $(V, \tilde{\phi}) \in D_{M_2}$ there exists $(U, \phi) \in D_{M_1}$ such that $f^{-1}(V)=U$. Since $f: M_1\longrightarrow M_2$ is $(r, \alpha)$--differentiable at $p$, we have $\alpha=\mu_2(V)=\mu_1(f^{-1}(V))$. Thus, $\alpha=\mu_2(V)=\mu_1(U)$. Now, Equation (\ref{5}) implies that $\alpha=\mu_1(U)=\mu_3(W)$.

Also, we have
\begin{equation*}
(gof)(U)=g(f(U))=g(V)\subseteq W;
\end{equation*}
\begin{align*}
\mu_1((gof)^{-1}(W)) &=\mu_1(f^{-1}(g^{-1}(W))) =\mu_1(f^{-1}(V)) \\&=\mu_2(V) =\mu_2(g^{-1}(W)) = \mu_3(W);
\end{align*}
and
\begin{equation*}
\mu_3o(gof)=(\mu_3og)of=\mu_2of=\mu_1.
\end{equation*}

Finally, the mapping
\begin{equation*}
\tilde{\tilde{\phi}}o(gof)o\phi^{-1}=(\tilde{\tilde{\phi}}ogo\tilde{\phi}^{-1})o(\tilde{\phi}ofo\phi^{-1}): \phi(U)\longrightarrow \tilde{\tilde{\phi}}(W)
\end{equation*}
is a $C^r$--map at $\phi(p)$, since $\tilde{\tilde{\phi}}ogo\tilde{\phi}^{-1}$ and $\tilde{\phi}ofo\phi^{-1}$ are $C^r$--maps.
\end{proof}

Let $(M_i, D_{M_i})$ be a $C^n$ $\mu_i$--selective Banach manifold over a Banach space $E_i$ for $i \in \{1, 2\}$.

\begin{definition}\label{item:3.1.}
Suppose that $f: M_1\longrightarrow M_2$ is an $(r, \alpha)$--differentiable map. $f$ is called $(r, \alpha)$--diffeomorphism if $f$ is bijective, and $f^{-1}:M_2 \rightarrow M_1$ is also $(r, \alpha)$--differentiable.
\end{definition}

\begin{theorem}
Let $(M_i, D_{M_i})$ be a $C^n$ $\mu_i$--selective Banach manifold over a Banach space $E_i$ for $i\in \{1, 2, 3\}$, $f: M_1\longrightarrow M_2$ be an $\alpha$--smooth diffeomorphism, and $g:M_2\longrightarrow M_3$ be any map. Then, $g$ is $\alpha$--smooth if and only if $gof$ is $\alpha$--smooth.
\end{theorem}

\begin{proof}
If $g:M_2\longrightarrow M_3$ is $\alpha$--smooth, then $gof$ is $\alpha$--smooth by Theorem \ref{item: 3 7}.

Now, suppose that $gof: M_1 \longrightarrow M_3$ be $\alpha$--smooth, and $q$ be any point in $M_2$. Since\\ $f: M_1\longrightarrow M_2$ is surjective, there exists $p \in M_1$ such that $f(p)=q$. Put $g(q)=m$. It follows from $\alpha$--smoothness of $g o f$ that $K_{gof, \alpha}\neq \phi$. Thus, there exists $(U,\phi; W, \eta) \in K_{gof, \alpha}$ such that $(U, \phi) \in D_{M_1}$ around $p$, and $(W, \eta) \in D_{M_3}$ around $m$, where $(gof)(U) \subseteq W$, and $\mu_1(U)=\mu_3(W)=\alpha$.

Considering Theorem \ref{item: 3 6}, for the mapping $f^{-1}: M_2\longrightarrow M_1$, and for the chart $(U, \phi) \in D_{M_1}$, $(f^{-1})^{-1}(U)=V \text{ i.e. } f(U)=V$ for some chart $(V, \psi) \in D_{M_2}$ around $q \in M_2$. Thus,
\begin{equation*}
g(V)=g(f(U))=(gof)(U)\subseteq W.
\end{equation*}

Since $f^{-1}: M_2\longrightarrow M_1$ is $\alpha$--smooth, for the chart $(U, \phi) \in D_{M_1}$ with $\mu_1(U)=\alpha$, we have
\begin{equation*}
\mu_1(U)=\mu_2((f^{-1})^{-1}(U))=\mu_2(f(U))=\mu_2(V).
\end{equation*}
Therefore, $K_{g, \alpha}\neq \emptyset$.

The $\alpha$--smoothness of $f: M_1\longrightarrow M_2$ and $gof: M_1\longrightarrow M_3$ result in $\mu_2 o f=\mu_1$ and $\mu_3 o(gof)=\mu_1$, respectively. Thus, $\mu_3 o(gof)=\mu_2 o f$. Since $f$ is invertible, $\mu_3 o g=\mu_2$. Therefore, Definition \ref{3 1} (\ref{a2}) is satisfied by $g: M_2\longrightarrow M_3$.

Let $(W, \eta) \in D_{M_3}$, and $\mu_3(W)=\alpha$. Considering $f: M_1\longrightarrow M_2$ is $\alpha$--smooth, we get
\begin{equation}\label{6}
\mu_2 o f=\mu_1.
\end{equation}
It follows from invertibility of $f$ and Equation (\ref{6}) that
\begin{equation}\label{7}
\mu_2=\mu_1 o f^{-1}.
\end{equation}
Now, considering $gof: M_1\longrightarrow M_3$ is $\alpha$--smooth, and using Equation (\ref{7}), we obtain
\begin{align*}
\alpha=\mu_3(W)=\mu_1((gof)^{-1}(W))= \mu_1(f^{-1}(g^{-1}(W)))=\mu_2(g^{-1}(W)).
\end{align*}
Therefore, Definition \ref{3 1} (\ref{a3})  is satisfied by $g: M_2\longrightarrow M_3$.

Finally, we inspect whether the Definition \ref{3 1} (\ref{a4}) is satisfied by $g: M_2\longrightarrow M_3$.
\begin{align*}
\eta o g o \psi^{-1} & =\eta o g o (f o f^{-1}) o \psi^{-1} \\& =(\eta o g o f)o(\phi^{-1}o\phi)o(f^{-1}o\psi^{-1}) \\& =[\eta o (gof) o \phi^{-1}]o(\phi o f^{-1} o \psi^{-1}).
\end{align*}
The mappings $\eta o (gof) o \phi^{-1}$ and $\phi o f^{-1} o \psi^{-1}$ are smooth since $gof$ and $f^{-1}$ are $\alpha$--smooth. Therefore, the mapping $\eta o g o \psi^{-1}$ is smooth.
\end{proof}

\section{The tangent space of selective Banach manifolds}

In this section, we present the notion of the tangent space to a $\mu$--selective Banach manifold at a given point, and study its properties. Through out this section, we assume that every vector space is defined over the field $F$, where $F=\mathbb{R}$, or $F=\mathbb{C}$.

Assume that $(M, A, \mu)$ is a $C^n$ selective Banach manifold modeled over the Banach space $E$, and $n\geq 1$. For every $p \in M$ and $\alpha \in Im\mu$, we define an $(r, \alpha)$--differentiable multi-path through $p$ by a multi-function `$\gamma: (-1, 1) \longrightarrow M$' satisfying the following conditions:
\begin{enumerate}[label=d\arabic*),itemindent=1cm,ref=d\arabic*]
\item $\gamma(0)=p$;
\item $\phi o \gamma: (-1, 1) \longrightarrow E$ is a $C^r$ map for all $(U, \phi) \in A$, where $U \bigcap \gamma((-1, 1))\neq\emptyset$, $\mu(U)=\alpha$, and $r\geq 1$.
\end{enumerate}
We denote the set of all $(r, \alpha)$--differentiable multi-paths through $p$ by $W^{p, \alpha}$.

We define the relation `$\sim$' by
\begin{equation*}
\gamma \sim \beta\quad \Longleftrightarrow \quad \dfrac{d(\phi o \gamma)}{dt}(0)=\dfrac{d(\phi o \beta)}{dt}(0);
\end{equation*}
where $U \bigcap \gamma((-1, 1))\neq\emptyset$, $U \bigcap \beta((-1, 1))\neq\emptyset$, and $\mu(U)=\alpha$.
This relation is an equivalence relation.

\begin{definition}
We denote $\dfrac{W^{p, \alpha}}{\sim}$ by $T_p^{\mu, \alpha}(M)$ and call it the tangent space of level $\alpha$ to $M$ at $p$.
\end{definition}

Let $\gamma$ be an $(r, \alpha)$--differentiable multi-path through $p$, $(U, \phi) \in A$, $U \cap \gamma((-1, 1))\neq \emptyset$, and $\mu(U)=\alpha$. We define the $j$th component of $\gamma$ by    \begin{equation*}
\gamma_j^U: (-1, 1) \longmapsto U \quad \text{ with } \quad \gamma_j^U(t)=\gamma(t) \cap U;
\end{equation*}
where $t \in \gamma^{-1}(U)$, $j \in J$, and $CardJ< \infty$.

Let $r\geq 1$. The mapping $\gamma_j^U$ is a $C^r$ map for every $j \in J$.

Restricting $\sim$ to $U$, we have $[\gamma_j^U] \in T_p(U)$. Put
\begin{equation*}
A^{p, \alpha}=\{(U, \phi)\mid (U, \phi) \in A \quad \& \quad \mu(U)=\alpha \quad \& \quad p \in U\}.
\end{equation*}
Now, define: $(U_1, \phi_1) \sim' (U_2, \phi_2)$ iff the following conditions are satisfied:
\begin{enumerate}[label=e\arabic*),itemindent=1cm,ref=e\arabic*]
\item $\phi_i(U_1 \bigcap U_2)$ is a Banach embedded sub-manifold of $\phi_i(U_i)$;
\item there exists a Banach diffeomorphism $f_i: \phi_i(U_1 \bigcap U_2) \longrightarrow \phi_i(U_i)$;
\end{enumerate}
for $i \in \{1, 2\}$.

\begin{theorem}
There exists a one to one correspondence $f: T_p^{\mu, \alpha}(M) \longrightarrow \prod_{[U, \phi]}T_P(U)$ that gives a vector space structure to $T_p^{\mu, \alpha}(M)$ by transferring the structure of $\prod_{[U, \phi]}T_P(U)$ to $T_p^{\mu, \alpha}(M)$.
\end{theorem}

\begin{proof}
Define the mapping $f$ by
\begin{equation*}
f: T_p^{\mu, \alpha}(M) \longrightarrow \prod_{[U, \phi]}T_P(U)
\end{equation*}
\begin{equation*}
f([\gamma]) = \prod_{[U, \phi]}[\gamma_j^U].
\end{equation*}
Clearly, $f$ is one to one and onto. Let $v, w \in T_p^{\mu, \alpha}(M)$ be given. Then, $T_p^{\mu, \alpha}(M)$ endowed with the operations `$+$' and `$.$'
\begin{equation*}
v+w:=f^{-1}((f(v)+f(w));
\end{equation*}
\begin{equation*}
c.v:=f^{-1}(c.f(v));
\end{equation*}
is a vector space.

In fact, for $(v, w),(v',w') \in T_p^{\mu, \alpha}(M) \times T_p^{\mu, \alpha}(M)$ with $v=v'$ and $w=w'$, we have
\begin{equation*}
v+w=f^{-1}(f(v)+f(w))=f^{-1}(f(v')+f(w'))=v'+w';
\end{equation*}
\begin{equation*}
c.v=f^{-1}(c.f(v))=f^{-1}(c.f(v'))=c.v'.
\end{equation*}
Thus, the operations `$+$' and `$.$' are well-defined. It is easy to check that the other conditions are held.
\end{proof}

\begin{definition}
Let $(M_i, D_{M_i})$ be a $C^n$ $\mu_i$--selective Banach manifold for $i \in \{1, 2\}$, $\alpha \in Im\mu_i$, and $f: M_1 \longrightarrow M_2$ be an $(r, \alpha)$--differentiable mapping at $p \in M_1$. We denote the $\alpha$--level differential of $f$ at $p$ by $d_p^{\alpha}f$, and define it as the linear mapping
\begin{equation*}
d_p^{\alpha}(f): T_p^{\mu_1, \alpha}(M_1) \longrightarrow T_{f(p)}^{\mu_2, \alpha}(M_2)
\end{equation*}
\begin{equation*}
[\gamma] \longmapsto [f o \gamma];
\end{equation*}
where $\gamma: (-1, 1) \longrightarrow M_1$, and $\gamma(0)=p$.
\end{definition}

The observers have an essential role in $(r, \alpha)$--differentiability of mappings between $\mu$--selective Banach manifolds so that a mapping that is $(r, \alpha)$--differentiable with respect to a given observer may not be $(r, \alpha)$--differentiable with respect to another choice of observer. As an example, for the constant mapping $f\equiv c$, consider the observer for which $K_{f, \alpha}=\emptyset$. In this case, $f$ is not $(r, \alpha)$--differentiable.

\begin{theorem}
Let $(M_i, D_{M_i})$ be a $C^n$ $\mu_i$--selective Banach manifold for $i \in \{1, 2\}$, and $f: M_1 \longrightarrow M_2$ is a constant mapping. If $f$ is $(r, \alpha)$--differentiable, then $d_p^{\alpha}f=0$ for all $p \in M_1$.
\end{theorem}

\begin{proof}
Since $f$ is $(r, \alpha)$--differentiable, then $K_{f, \alpha}\neq\emptyset$. Thus, $d_{\phi(p)}(\psi o f o \phi^{-1})=0$ for every $(U, \phi; V, \psi) \in K_{f, \alpha}$. Therefore, for all $p \in M_1$, $d_p^{\alpha}f=0$.
\end{proof}

\section{Conclusion}
In this paper, the concept of the $\mu$--selective Banach manifold is introduced, and it's fundamental properties are studied. Specifically, it is proved that the product of two selective Banach manifolds is a selective Banach manifold. Next, the concept of the $\alpha$--level differentiation of the mappings between selective Banach manifolds is presented and a version of chain rule theorem is given for the mappings between $\mu$--selective Banach manifolds. Moreover, the notion of the tangent space of a selective Banach manifold at a given point is studied.

It will be interesting to establish a version of Inverse Function Theorem for the $\mu$--selective Banach manifolds for further research.

%%%%%%%%%%%%%%%%%%%%%%%%%%%%%%%%%

\end{document}